\documentclass[a4paper, 12pt, reqno]{amsart}
\usepackage{amssymb, amsfonts, amsmath, bm}
\usepackage[all]{xy}
\usepackage[shortlabels]{enumitem}
\usepackage{hyperref, color}
\usepackage{commath}
\usepackage{esdiff}
\usepackage{graphicx,subfigure}
\usepackage[center]{caption}
\usepackage{amsthm}
\usepackage{diagbox}
\usepackage{multirow}
\usepackage{esdiff}
\usepackage[pagewise]{lineno}
\usepackage[utf8]{inputenc}
\usepackage{soul}
\usepackage{comment} 
\usepackage{pifont}
\usepackage{hyperref, color, xcolor}
\hypersetup{
	colorlinks = true,
	linkcolor = blue,
	filecolor = blue,
	urlcolor = red,
	citecolor = red
}
\usepackage{cite}
\usepackage{lineno}
\newcommand{\inner}[2]{\left\langle #1,#2\right\rangle}

\usepackage[nameinlink,capitalise]{cleveref} 

\usepackage[a4paper,twoside,top=1in, bottom=1in, left=0.8in, right=0.8in]{geometry}

\DeclareMathOperator{\Ric}{Ric}

\newtheorem{theorem}{Theorem}[]
\newtheorem{lemma}[theorem]{Lemma}
\newtheorem{proposition}[theorem]{Proposition}
\newtheorem{corollary}[theorem]{Corollary}
\newtheorem*{conjecture*}{Conjecture}

\newtheorem{theo}{Theorem}

\newtheorem{cor}[theo]{Corollary}
\newtheorem{prop}[theo]{Proposition}

\theoremstyle{definition}
\newtheorem{example}{Example}
\newtheorem*{remark*}{Remark}

\newtheorem{remark}[theorem]{Remark}

\numberwithin{figure}{section}

\usepackage{graphicx} 
\title[On Generalized $m$-Quasi-Einstein Manifolds]{On Generalized $m$-Quasi-Einstein Manifolds}

\author[Alcides de Carvalho, Anderson Lima and W. O. Costa-Filho]{Alcides de Carvalho, Anderson Lima and W. O. Costa-Filho}

        \address{Universidade Federal de Pernambuco\\
		Departamento de Matematica\\
		 Av. Jorn. Aníbal Fernandes, s/n, Cidade Universitária, Recife - PE, Brazil, CEP:  50740-560.}	
	\email{alcides.junior@ufpe.br}
  \address{
  CPMAT - IM, Universidade Federal de Alagoas, Maceió, AL, 57072-970, Brazil}
\email{jose.lima@im.ufal.br}
  
\address{
  Universidade Federal de Alagoas, Campus Arapiraca, Av. Manoel Severino Barbosa, S/N, Bom Sucesso, 
 Arapiraca - AL, Brazil, CEP: 57309-005.}
\email{wagner.filho@arapiraca.ufal.br}

\date{\today}
 
\begin{document}
\keywords{m-Quasi-Einstein manifolds · Killing vector fields · Scalar curvature}

\subjclass{Primary 53C25; Secondary 53C21, 53C65}

\begin{abstract}
In this paper, we study generalized $m$-quasi-Einstein manifolds $(M^n,g,X,\lambda)$ under suitable constraints on the potential vector field $X$. We establish that, under suitable integral assumptions, the potential vector field is a Killing vector field, extending earlier results of Sharma \cite{Rsharma2025} to the generalized setting. We derive consequences under sign conditions on $m$ and $\lambda$, including triviality results. We also revisit a recent theorem by Ghosh \cite{ghosh}, discuss a subtle issue in the argument, and provide a new formulation and proof. Finally, we obtain rigidity results when $X$ is a geodesic vector field.
\end{abstract}

\maketitle
\section{Introduction}

Besides their geometrical importance, Riemannian manifolds admitting an Einstein-like structure are of interest in Mathematical Physics where they are usually referred as quasi-Einstein metrics. In recent years, several generalizations of this notion have been introduced and extensively investigated in the literature. Among them, the class of generalized m-quasi-Einstein manifolds has attracted considerable attention due to its rich geometric nature and its connections with Ricci solitons, warped product Einstein metrics and many gravitational contexts in spacetime. See, for instance, Bahuaud, Gunasekaran, Kunduri, and Woolgar in \cite{Bahuaud2024}, as well as Wylie in \cite{Wylie2023} and the relevant references therein.

An $n$-dimensional Riemannian manifold $(M^n, g)$, with $n \geq 2$, is said to be a \emph{generalized $m$-quasi-Einstein manifold} for a nonzero constant $m \in \mathbb{R}$ if there exists a smooth vector field $X$ on $M^n$ and a smooth function $\lambda \colon M \to \mathbb{R}$ such that the following equation holds:
\begin{equation}\label{eq:quasi-einstein}
\operatorname{Ric} + \frac{1}{2} \mathcal{L}_X g - \frac{1}{m} X^\flat \otimes X^\flat = \lambda g,
\end{equation}
where $\operatorname{Ric}$ denotes the Ricci tensor, $\mathcal{L}_X g$ is the Lie derivative of the metric $g$ in the direction of $X$, $X^\flat$ is the 1-form metrically dual to $X$, and $\otimes$ denotes the tensor product. Throughout this paper, we adopt the convention that $(M^n,g,X,\lambda)$ is called an \emph{$m$-quasi-Einstein manifold} for solutions of \eqref{eq:quasi-einstein} with $\lambda$ constant.

This notion was introduced by Barros and Gomes in \cite{BarrosGomes} and was also considered by Barros and Ribeiro in \cite{BarrosRibeiro2014}, when $X$ is the gradient of a smooth potential function on $M^n.$ More recently, these equations have been referred to as the \emph{generalized extremal horizon equations}. See Definition~1.4 in Colling and Dunajski
\cite{colling2026quasi}, where several related results
are established. In black hole geometry, the vacuum
near-horizon equations governing the intrinsic geometry
of cross-sections of degenerate Killing horizons take
the form of the $2$-quasi-Einstein equation with constant
$\lambda$ (see \cite{KunduriLucietti2013,Bahuaud2024}
and also Lewandowski and Kami\'nski
\cite{kaminski2024extreme}). In four-dimensional
Einstein--Maxwell theory, the electromagnetic contribution
allows $\lambda$ to vary over the two-dimensional horizon
cross-section. The corresponding near-horizon equations
thus give a special case of the generalized
$2$-quasi-Einstein equation \eqref{eq:quasi-einstein}
(see Colling, Katona and Lucietti
\cite{colling2025rigidity}).

Generalized $m$-quasi-Einstein manifolds are special generalizations of Ricci solitons, and a main motivation for studying this class of manifolds is their connection with general relativity. In addition, when $X=\nabla f$ is a gradient vector field, $\lambda$ is constant and $m$ is a positive integer, an $m$-quasi-Einstein manifold is precisely the base of an $(n+m)$-dimensional Einstein warped product with $m$-dimensional Einstein fiber. This connection was first observed by Kim and Kim \cite{Kim2003} (see also Case, Shu and Wei \cite{case2011}).
For more details on this subject, including some appropriate examples and important characterization results, we recommend that the reader consult \cite{BarrosGomes, BarrosRibeiro2014} and Ghosh in \cite{ghosh,Gho20}. A generalized $m$-quasi-Einstein manifold is called \textit{trivial} when $X$ is identically zero. In what follows, $S$ denotes the scalar curvature of $M^n$, defined as the trace of $\operatorname{Ric}$.
 
Given their significance, a number of results have been derived with regard to quasi-Einstein manifolds under various assumptions. In \cite{bahuaud2022static} Bahuaud et al. studied the case where the dual one-form to $X$ is closed. In \cite{Bahuaud2024}, the authors considerably addressed the solutions of \eqref{eq:quasi-einstein} with the vector field $X$ divergence-free but not identically zero. In particular, they proved that if $M^n$ is closed and $m \neq -2$, then $X$ is a Killing vector field. As a consequence, the scalar curvature $S$ is constant.

More recently, Cochran \cite{Cochran2025} and Costa-Filho \cite{CostaFilho2024} independently showed that, on a closed $m$-quasi-Einstein manifold $(M^n,g,X,\lambda)$, the scalar curvature $S$ is constant if and only if $X$ is a Killing vector field. It is worth emphasizing that Cochran’s result requires the additional assumption $m \neq -2$. 
As Cochran pointed out in \cite{cochran2026compact}, this result was first shown by Ghosh in \cite{Gho20}. The implication that constant scalar curvature implies that $X$ is a Killing vector field is Theorem~4.2 of \cite{Gho20}. Furthermore, equation~(5.18) in \cite{Gho20} yields the converse implication.

Motivated by these developments, R. Sharma \cite{Rsharma2025} obtained a closely related rigidity result under a weaker hypothesis. More precisely, Theorem~1.1 in \cite{Rsharma2025} states that if $(M^n,g,X,\lambda)$ is a closed $m$-quasi-Einstein manifold, then $X$ being a Killing vector field implies that the scalar curvature $S$ is constant. Conversely, if
$ 
\int_M (\mathcal{L}_X S)\, dM\le 0,
$ in particular, if $S$ is constant along the flow of $X$ then $X$ must be a Killing vector field. At the time of publication, neither \cite{Cochran2025}, \cite{CostaFilho2024}, nor Sharma’s work were aware of Ghosh’s earlier result.

Inspired by Sharma’s result, we prove the following rigidity statement under analogous integral assumptions in the generalized $m$-quasi-Einstein setting.

\begin{theo}\label{theoA}
Let $(M^n,g,X,\lambda)$ be a closed generalized $m$-quasi-Einstein manifold. Assume that
$$
\int_M \langle X,\nabla S\rangle\, dM \le 0
\quad \text{and} \quad
\int_M \langle X,\nabla \lambda\rangle\, dM \ge 0.
$$
Then $X$ is a Killing vector field.
\end{theo}

\begin{remark}
If $\lambda$ is constant, then the second integral vanishes and the result reduces to the corresponding theorem of Sharma.
\end{remark}

The following result corresponds to the divergence-free-to-Killing implication in Ghosh’s Theorem~6.1~\cite{ghosh}. It also follows immediately from Theorem~A.

\begin{prop}[Ghosh {\cite[Theorem 6.1]{ghosh}}]\label{prop-div}
    Let $ (M^n, g, X,\lambda)$ be a closed generalized $ m $-quasi-Einstein manifold. If $X$ is divergence-free, then $X$ is a Killing vector field.
\end{prop}

In fact, Theorem~6.1 of \cite{ghosh} asserts more.

\begin{theorem}[Ghosh, Theorem 6.1]\label{thm:ghosh61}
Let $(M^n, g, X, \lambda)$ be a closed generalized $m$-quasi-Einstein manifold. If $X$ is divergence-free, then $X$ is a Killing vector field of constant length, and the scalar curvature $S$ is constant.
\end{theorem}
The first assertion of Theorem~\ref{thm:ghosh61} is established in \cite{ghosh} by integrating the Bochner-type identity (6.5) of \cite{ghosh} (Lemma~\ref{lem:laplacianX} below), whose $\nabla\lambda$-term integrates to zero when $\operatorname{div}X=0$, and by comparing the result with Yano's formula. This argument does not use the constancy of $\lambda$ and is valid in the generalized setting; it is reproduced in the proof of Proposition~\ref{prop-div} in Section~\ref{sec:proofs}. The remaining assertions of Theorem~\ref{thm:ghosh61} are then derived in \cite{ghosh} from identity (5.3) there, which is the $m$-quasi-Einstein version of Lemma~\ref{eq10} below. When $\lambda$ is not constant, that identity acquires the additional term $-(n-2)\langle\nabla\lambda,Z\rangle$, see \eqref{neweq10.1}. Consequently, for $n\geq3$ the argument of \cite{ghosh} no longer shows that $|X|$ is constant, and this assertion becomes equivalent to the constancy of $\lambda$ (Theorem~\ref{teo6.1} below). For $n=2$ the identity is unchanged and the argument does show that $|X|$ is constant (indeed $X\equiv0$, see Theorem~\ref{teo6.1}; what fails is the final step, since the constancy of $S$ follows from that of $|X|$ only through the trace identity \eqref{eq2}, which requires $\lambda$ to be constant). The precise statement in the generalized setting is the following.

\begin{theo}\label{teo6.1}
Let $(M^n,g,X,\lambda)$ be a generalized $m$-quasi-Einstein manifold
with $X$ Killing. Then
$$
\nabla_X X
=
\frac{m(n-2)}{4}\nabla\lambda,
\qquad
\nabla S
=
\frac{n+2}{2}\nabla\lambda,
\qquad
\nabla |X|^2
=
-\frac{m(n-2)}{2}\nabla\lambda.
$$
In particular, $S$ is constant if and only if $\lambda$ is constant.
For $n\geq 3$, these conditions are also equivalent to $|X|$ being
constant, and to $\nabla_X X=0$.
For $n=2$, necessarily $X=0$ and
$ 
\lambda=\frac{S}{2};
$ the functions $S$ and $\lambda$ may be nonconstant.
\end{theo}

The following example shows that the original statement of Theorem 6.1 in~\cite{ghosh} is false in the generalized setting, already in dimension two.

\begin{example}\label{ex1}
Let $(M^2,g)$ be any closed surface whose Gaussian curvature $K$ is not constant. Set $X\equiv0$, $\lambda:=K$. Then $X$ is a Killing field and
$$
\Ric+\frac12\mathcal{L}_Xg-\frac1mX^\flat\otimes X^\flat=\Ric=K\,g=\lambda g.
$$

Since $K$ is not constant, $\lambda$ is nonconstant. Hence $(M^2,g,X,\lambda)$ is a closed generalized $m$-quasi-Einstein manifold with $\lambda$ nonconstant and $X$ Killing, and $S=2K$ is not constant. Note that in this example $X$ is divergence-free and has constant length, so that the only conclusion of Theorem~\ref{thm:ghosh61} which fails is the constancy of $S$. By Theorem~\ref{teo6.1}, every closed generalized $m$-quasi-Einstein surface with $\operatorname{div}X=0$ is of this form.
\end{example}

\begin{remark}\label{rem3}
In the proof of Theorem 6.1 in \cite{ghosh}, the assumption that $M^n$ is closed is used only to conclude that $\operatorname{div}X=0$ implies that $X$ is a Killing vector field (see Proposition~\ref{prop-div}). The identities \eqref{eq104} and \eqref{eq:div} used afterwards hold on any generalized $m$-quasi-Einstein manifold with Killing potential, and, for $n\geq 3$, they yield the constancy of $|X|$, $\lambda$ and $S$ as soon as one of these three functions is constant (see the proof of Theorem~\ref{teo6.1}). For $n=2$, this implication does not hold, since $|X|\equiv 0$ may occur while $\lambda$ and $S$ are nonconstant. However, none of them needs to be constant when $M^n$ is not compact. Indeed, the following example exhibits a noncompact generalized $m$-quasi-Einstein manifold with a Killing vector field $X$ for which $|X|$, $\lambda$ and $S$ are all nonconstant.
\end{remark}
\begin{example}
Consider the Riemannian manifold
$
(M^3,g)=\bigl((0,\infty)\times\mathbb{R}^2,\,dx^2+x^{-2}dy^2+x^{-4}dz^2\bigr),
$
the vector field
$
X=\partial_y,
$ on $M$,
the real constant
$
m=\frac14,
$
and the smooth function
$
\lambda=-\frac{8}{x^2}$ of $M$.
Since the metric coefficients do not depend on the variable $y$, the vector field $X$ is Killing. Moreover, a direct computation yields
$\Ric=-\frac{8}{x^2}\,dx^2-\frac{4}{x^4}\,dy^2-\frac{8}{x^6}\,dz^2,
$ and
$X^\flat\otimes X^\flat=x^{-4}\,dy^2.
$
Since $\mathcal{L}_Xg=0$, we obtain
$
\Ric+\frac12\mathcal{L}_Xg-\frac1mX^\flat\otimes X^\flat
=-\frac{8}{x^2}\,g.
$
Therefore, $(M^3,g,X,\lambda)$ satisfies the generalized $m$-quasi-Einstein equation \eqref{eq:quasi-einstein}.
Furthermore, the scalar curvature is
$
S=-\frac{20}{x^2},
$ not constant.
Here $|X|^2=x^{-2}$ is nonconstant as well.

\end{example}

A smooth vector field $Y$ on a Riemannian manifold $(M^n,g)$ is called a \textit{geodesic vector field} if $\nabla_YY=0$, where $\nabla$ is the Levi-Civita connection on $M^n$. Examples of geodesic vector fields include Killing vector fields of constant length on a Riemannian manifold. Recall that a vector field $Y$ on $M^n$ is a geodesic vector field if and only if each integral curve of $Y$ is an affinely parametrized geodesic in $M^n.$ Physically, geodesic vector fields naturally appear in many important situations (see, for example, \cite{deshmukh2019geodesic,yano1961geodesic}). In general relativity, in turn, the called \textit{surface gravity} for a Killing horizon is defined by the fact that the Killing vector is a non-aﬃnely parameterized geodesic on the Killing horizon (see \cite{bardeen1973b}).

Recall also that a smooth map $\varphi\colon(M^n,g)\to(N,h)$ between Riemannian manifolds is \emph{harmonic} if it is a critical point of the energy functional $E(\varphi)=\frac12\int_M|d\varphi|^2\,dM$, or equally if its tension field $\tau(\varphi)=\operatorname{tr}_g(\nabla d\varphi)$ vanishes identically (see \cite{xin2012geometry}). A diffeomorphism between open subsets of $(M^n,g)$ is called a \emph{local harmonic diffeomorphism} if it is a harmonic map with respect to the metric $g$ on both its domain and its target. Now, following Stepanov and Shandra in \cite{StepanovShandra2003}, we say that a smooth vector field $Y$ on $(M^n,g)$ is an \textit{infinitesimal harmonic transformation} if the one-parameter group of local transformations of $M^n$ generated by $Y$ consists of local harmonic diffeomorphisms. Equivalently, $\operatorname{tr}_g(\mathcal{L}_Y\nabla)=0$ (see \cite{StepanovShandra2003,nouhaud1972transformations} and Section~\ref{sec:prelim}). Killing vector fields are infinitesimal harmonic transformations.

As a consequence of Lemma~\ref{eq10} and also of the linearized scalar curvature equation, we obtain the following version of Theorem~5.1 in \cite{ghosh} in the generalized setting.

\begin{theo}\label{D}
Let $(M^n,g,X,\lambda)$ be a generalized $m$-quasi-Einstein manifold with geodesic potential vector field $X$. Assume that $X$ is an infinitesimal harmonic transformation. If $n=2$, then $X$ is Killing. If $n\geq 3$, then $X$ is Killing if and only if $\lambda$ is invariant along the flow of $X$, that is,
$ 
X(\lambda)=0.
$
\end{theo}

In \cite{CostaFilho2024}, the third author presented a proof of Proposition 2.4 in \cite{Bahuaud2024} (see page 8) as a consequence of Lemma 1 in \cite{BarrosRibeiro2014}. In this work, we extend this approach to the generalized setting by establishing an analog of Lemma 1 from \cite{BarrosRibeiro2014} in this context (see Lemma~\ref{lem:laplacianX}). As an application, we derive the following fact.

\begin{cor}\label{corE}
Let $(M^n, g, X, \lambda)$ be a closed generalized $m$-quasi-Einstein manifold.
Assume that $X$ is divergence-free.
If $m < 0$ and $\lambda \leq 0$ on $M$, then $X$ vanishes identically.
\end{cor}

In the same work \cite{ghosh}, Ghosh proved that if $(M^n,g,X,\lambda)$ is a closed $m$-quasi-Einstein manifold with geodesic potential $X$ and non-positive Ricci tensor, then $X=0$ and $M^n$ is Ricci-flat, or $X=0$ and $M^n$ is Einstein, or
$
\lambda = -\frac{1}{m}|X|^2.
$
We extend this result from the $m$-quasi-Einstein setting to the class of generalized $m$-quasi-Einstein manifolds.

\begin{theo}\label{theoF}
Let $(M^{n},g,X,\lambda)$ be a closed generalized $m$-quasi-Einstein manifold with geodesic potential vector field $X$. Assume that  $\int_M \operatorname{Ric}(X,X)\, dM \le 0$ and
$\int_M \langle X,\nabla \lambda\rangle\, dM \le 0$.
Then $X$ is parallel. If, in addition, $n>2$, then $\lambda$ is constant. Moreover, the following statements hold:
\begin{enumerate}   
    \item  If $\lambda =0$, then $X=0$ and $(M^n,g)$ is Ricci-flat;
    \item If $\lambda m>0$ on $M^n$, then $X=0$ and $(M^n,g)$ is Einstein;
    \item If $X$ is nonzero at some point of $M$, then
$ 
\lambda=-\frac1m|X|^2.
$
\end{enumerate}
\end{theo}


We close this introduction with explicitly closed nontrivial examples and indicate which hypotheses of Theorems A, D, and F they satisfy. We first recall that, for a closed generalized $m$-quasi-Einstein manifold with conformal potential field (see~\cite[Theorem~A]{CarvalhoCosta2026integral}), we have 
\begin{equation}\label{cor6(2)}
\int_M \langle X,\nabla\lambda\rangle\,dM
=
-\frac1n\int_M(\operatorname{div}X)^2\,dM\leq0.
\end{equation}

Hence, equality holds if and only if $X$ is Killing.

The strict case actually occurs. 
\begin{example}
For $m>0$, the construction in
\cite[Appendix, Theorem~F]{CarvalhoCosta2026integral}
provides a closed generalized $m$-quasi-Einstein manifold
$$
M^n=S^1_L\times F,\qquad
\bar g=dt^2+\varphi(t)^2g,\qquad
X=\varphi(t)\partial_t,
$$
where $(F^{n-1},g),\, n\geq 3,$ is closed Einstein with
$\operatorname{Ric}_g=-(n-2)g$, and $\varphi$ is a positive,
nonconstant $L$-periodic solution of the corresponding
warped-product equation. The function $\lambda$ is
nonconstant, and
$$
\mathcal L_X\bar g=2\varphi'\bar g,
\qquad
\operatorname{div}X=n\varphi'.
$$

Thus, $X$ is conformal but not Killing. Moreover,
$$
\int_M\langle X,\bar\nabla\lambda\rangle_{\bar g}\,dM
=-\frac1n\int_M(\operatorname{div}_{\bar g}X)^2\,dM
=-n\operatorname{Vol}(F,g)
\int_0^L(\varphi')^2\varphi^{n-1}\,dt<0.
$$

On the other hand, since $M$ is closed and $X$ is conformal, the
Kazdan--Warner identity (see
Bourguignon and Ezin \cite[Theorem~II.9]{bourguignon1987scalar}) yields
$$
\int_M\langle X,\bar\nabla S\rangle_{\bar g}\,dM=0.
$$
\end{example}

Thus, the hypothesis involving $\nabla\lambda$ in Theorem~A cannot be removed.

It is worth noting that the above example satisfies
$$
\int_M\langle X,\nabla\lambda\rangle\,dM<0,
\qquad
\int_M\operatorname{Ric}(X,X)\,dM>0,
$$
while $X$ is nonparallel, see
\cite[Appendix, Theorem~F]{CarvalhoCosta2026integral}. Thus, this
example shows that at least one of the two assumptions, namely, that
$X$ is geodesic or that
$$
\int_M \operatorname{Ric}(X,X)\,dM\le 0
$$
holds is necessary for the rigidity conclusion of
Theorem~\ref{theoF}. Notice, however, that this example does not
directly test the sharpness of Theorem~F, since its potential field
$X=\varphi(t)\partial_t$ is not geodesic:
$$
\nabla_X X=\varphi\varphi'\partial_t\not\equiv0.
$$
We also recall the spherical example of Barros and Ribeiro
\cite[Example~1]{BarrosRibeiro2014} to show that the hypothesis
$\int_M\langle X,\nabla\lambda\rangle\,dM\geq0$
cannot be omitted from Theorem~A.

\begin{example}[Standard spheres, {\cite[Example~1]{BarrosRibeiro2014}}]
Let $(\mathbb S^n,g)$, $n\ge2$, be the unit round sphere,
and let $h_v(x)=\langle x,v\rangle$ be the height function
associated with a fixed unit vector $v\in\mathbb R^{n+1}$.
For $\tau>1/n$ and $m\in\mathbb R\setminus\{0\}$, set
$$
u=\tau-\frac{h_v}{n},\qquad
f=-m\log u,\qquad
X=\nabla f,\qquad
\lambda=n-1-\frac{mh_v}{nu}.
$$

Since $\nabla^2h_v=-h_vg$, we obtain
$$
\nabla^2f-\frac1m\,df\otimes df
=-\frac{m}{u}\nabla^2u
=-\frac{mh_v}{nu}\,g.
$$

Together with $\operatorname{Ric}=(n-1)g$, this shows that
$(\mathbb S^n,g,X,\lambda)$ is a closed nontrivial generalized
$m$-quasi-Einstein manifold with constant scalar curvature
$S=n(n-1)$. Although the original construction is presented
for positive $m$, the same computation applies to all
real $m\neq0$.
\end{example}
In the preceding example,
$$
X=\frac{m}{nu}\nabla h_v,
\qquad
\nabla\lambda=-\frac{m\tau}{nu^2}\nabla h_v.
$$

Since $S$ is constant,
$$
\int_{\mathbb S^n}\langle X,\nabla S\rangle\,dM=0,
\qquad
\int_{\mathbb S^n}\langle X,\nabla\lambda\rangle\,dM
=-\frac{m^2\tau}{n^2}
\int_{\mathbb S^n}\frac{|\nabla h_v|^2}{u^3}\,dM<0.
$$

However, $X$ is not Killing, since a Killing gradient field
on a closed manifold has constant potential.
Thus, the hypothesis
$\int_M\langle X,\nabla\lambda\rangle\,dM\ge0$
cannot be omitted from Theorem~A, even when $S$ is constant.

According to the next example presented in \cite{Barros2014} a Berger sphere carries such a metric with a nontrivial vector field, which is not a gradient, defining a nontrivial $m$-quasi-Einstein structure.

\begin{example}[Berger spheres, {\cite[Example~1]{Barros2014}}]
Let $(S^3,g_{\kappa,\tau})$ be a Berger sphere with
$0<\kappa<4\tau^2$, and let $E_3$ denote the unit vertical
Killing vector field of the Hopf fibration. For $m>0$, set
$$
X=\sqrt{m(4\tau^2-\kappa)}\,E_3,
\qquad
\lambda=\kappa-2\tau^2.
$$

Since
$$
\operatorname{Ric}
=(\kappa-2\tau^2)g_{\kappa,\tau}
 +(4\tau^2-\kappa)E_3^\flat\otimes E_3^\flat,
$$
these data define a closed nontrivial $m$-quasi-Einstein
structure. Moreover, $X$ is Killing of constant length,
hence geodesic and an infinitesimal harmonic transformation.
Both $S$ and $\lambda$ are constant, so all the hypotheses
of Theorems~A and~D are satisfied.

On the other hand,
$$
\operatorname{Ric}(X,X)
=2m\tau^2(4\tau^2-\kappa)>0,
$$
and therefore $X$ is not parallel. Thus, this example satisfies
the geodesic assumption of Theorem~F and
$$
\int_{S^3}\langle X,\nabla\lambda\rangle\,dM=0,
$$
but not its conclusion. Consequently, the hypothesis
$\int_M\operatorname{Ric}(X,X)\,dM\leq0$ cannot be omitted
from Theorem~F, even when $\lambda$ is constant.
\end{example}

\section{Preliminaries}\label{sec:prelim}

In this section, we shall present some preliminaries which will be important for the establishment of the desired results. Throughout this paper, $(M^{n},g)$ denotes a connected and orientable Riemannian manifold.

We notice that taking the trace of equation \eqref{eq:quasi-einstein}, we arrive at  
\begin{equation}\label{eq2}
S + \operatorname{div}X = \frac{1}{m} |X|^2 + n\lambda.
\end{equation}

In what follows, we recall a key identity due to Barros and Ribeiro.
In the case of $m$-quasi-Einstein manifolds, that is, $\lambda$ is constant, it can be found in \cite[p.~215]{BarrosRibeiro2012}. Its extension to the generalized setting, with $X=\nabla f$, was obtained in \cite{BarrosRibeiro2014}. For the general case, this result was also proved by Ghosh in \cite[formula 6.5]{ghosh}. Here, $\Delta$ denotes the Laplacian operator on $M^n$.

\begin{lemma}\label{lem:laplacianX}
Let $(M^n,g,X,\lambda)$ be a generalized $m$-quasi-Einstein manifold. Then the following
identity holds
$$
\frac{1}{2}\,\Delta |X|^{2}
= |\nabla X|^{2}
- \mathrm{Ric}(X,X)
+ \frac{2}{m}\,|X|^{2}\,\mathrm{div}\,X
+ (2-n)\,\langle X,\nabla \lambda\rangle.
$$
\end{lemma}

When $\lambda$ is constant, the last term vanishes and the original formula is recovered. Lemma \ref{lem:laplacianX} yields the following result of rigidity under natural integral assumptions, which extends Theorem~1 of Barros and Ribeiro \cite{BarrosRibeiro2014} to the generalized setting. First, we recall that a smooth vector field $Y$ defined on a Riemannian manifold $(M^n,g)$ is said to be conformal if $\mathcal{L}_Y g=2\phi\, g$, for a smooth function $\phi$ on $M^n$. In particular, we see that $\phi$ satisfies $\mathrm{div}\,Y=n\phi.$ 
 
\begin{corollary}\label{cor6}
Let $(M^n,g,X,\lambda)$ be a closed generalized $m$-quasi-Einstein manifold with \\ $\int_M \langle X,\nabla \lambda\rangle dM \leq 0$. Then $X$ is a parallel vector field if any of the following conditions holds
\begin{enumerate}
\item $\displaystyle \int_M \mathrm{Ric}(X,X)\, dM 
\leq \frac{2}{m} \int_M |X|^2\mathrm{div}\,X \, dM$;

\item $X$ is conformal and $\displaystyle \int_M \mathrm{Ric}(X,X)\,dM\leq 0$;

\item $|X|$ is constant and $\displaystyle \int_M \mathrm{Ric}(X,X)\,dM\leq 0$.
\end{enumerate}
\end{corollary}

\begin{proof}
Integrating the expression of Lemma \ref{lem:laplacianX}, we obtain
\begin{equation}\label{eq:cor6}
\int_M |\nabla X|^2 dM = \int_M \mathrm{Ric}(X,X) dM - \frac{2}{m}\int_M |X|^2 \,\mathrm{div}\,X \,\,dM +(n-2)\int_M \langle X,\nabla \lambda\rangle dM.
\end{equation}
Since $n\geq2$ and $\int_M\langle X,\nabla\lambda\rangle\,dM\leq0$, the last term on the right-hand side of \eqref{eq:cor6} is non-positive. In case (1), the sum of the first two terms is non-positive by hypothesis. In case (3), $\int_M|X|^2\operatorname{div}X\,dM=|X|^2\int_M\operatorname{div}X\,dM=0$ by the divergence theorem, so the right-hand side reduces to $\int_M\operatorname{Ric}(X,X)\,dM+(n-2)\int_M\langle X,\nabla\lambda\rangle\,dM\leq0$. In case (2), write $\mathcal{L}_Xg=2\phi\,g$, so that $\operatorname{div}X=n\phi$ and $X(|X|^2)=(\mathcal{L}_Xg)(X,X)=2\phi|X|^2$. By the divergence theorem,
$$
n\int_M\phi|X|^2\,dM=\int_M|X|^2\operatorname{div}X\,dM=-\int_M X(|X|^2)\,dM=-2\int_M\phi|X|^2\,dM,
$$
hence $(n+2)\int_M\phi|X|^2\,dM=0$ and therefore $\int_M|X|^2\operatorname{div}X\,dM=0$, independently of the sign of $m$. The right-hand side of \eqref{eq:cor6} is then again $\int_M\operatorname{Ric}(X,X)\,dM+(n-2)\int_M\langle X,\nabla\lambda\rangle\,dM\leq0$. Hence, in each of the cases above, $\int_M |\nabla X|^2dM \leq 0,$ which implies $|\nabla X|=0$. Thus, $X$ is parallel.
\end{proof}
\begin{remark}
In case~(2) of Corollary~\ref{cor6}, the assumption
$
\int_M\langle X,\nabla\lambda\rangle\,dM\leq0
$
is in fact automatic. Indeed, for a conformal potential field, this follows directly from \eqref{cor6(2)}.
\end{remark}
In the $m$-quasi-Einstein setting, Bahuaud et al.~\cite{Bahuaud2024} proved rigidity results under the assumption $|X|^2 = -m\lambda > 0$ and for $2 < m \leq 4$, showing in particular that $X$ is parallel. This was later improved by the first and third named authors~\cite{CarvalhoCosta}, who showed that condition $|X|^2 = -m\lambda$ alone already implies that $X$ is parallel. We next study the corresponding statement in the generalized case.

\begin{proposition}\label{prop:parallel}
Let $(M^n,g,X,\lambda)$ be a closed generalized $m$-quasi-Einstein manifold such that
$
|X|^2 = -m\lambda.
$
Assume that one of the following conditions holds:
\begin{enumerate}
\item $m = 2(4-n)$;\\
\item $\displaystyle \frac{m}{2} + n - 4 > 0$ and 
$ 
\int_M \langle X, \nabla \lambda \rangle \, dM \leq 0;
$\\
\item $\displaystyle \frac{m}{2} + n - 4 < 0$ and 
$
\int_M \langle X, \nabla \lambda \rangle \, dM \geq 0.
$\\
\end{enumerate}
Then $X$ is parallel.
\end{proposition}
\begin{proof}
Proceeding as before, integrating the expression of Lemma \ref{lem:laplacianX}, we have that
\begin{equation*}
\int_M |\nabla X|^2 \, dM
= \int_M \mathrm{Ric}(X,X)\, dM
- \frac{2}{m}\int_M |X|^2 \,\mathrm{div}\,X \, dM
+ (n-2)\int_M \langle X,\nabla \lambda\rangle \, dM.
\end{equation*}

Combining \eqref{eq:quasi-einstein} with the identity $|X|^2 = -m\lambda$, we obtain
\begin{equation*}
\mathrm{Ric}(X,X) = \lambda |X|^2 + \frac{1}{m}|X|^4 
- \frac{1}{2}(\mathcal{L}_X g)(X,X)
= - \frac{1}{2}(\mathcal{L}_X g)(X,X).
\end{equation*}

Moreover,
$ 
(\mathcal{L}_X g)(X,X) = X(|X|^2) = -m\,\langle X,\nabla \lambda\rangle,
$ 
and hence
$ 
\mathrm{Ric}(X,X) = \frac{m}{2}\,\langle X,\nabla \lambda\rangle.
$ 

On the other hand, using the relation $|X|^2 = -m\lambda$  again  and the divergence theorem, we get
\begin{align*}
\int_M |X|^2 \,\mathrm{div}\,X \, dM
&= -m \int_M \lambda \,\mathrm{div}\,X \, dM \\
&= m \int_M \langle \nabla \lambda, X \rangle \, dM.
\end{align*}

Substituting these identities into the previous expression gives
\begin{equation*}
\int_M |\nabla X|^2 \, dM
= \frac{m}{2}\int_M \langle X,\nabla \lambda\rangle \, dM
- 2 \int_M \langle X,\nabla \lambda\rangle \, dM
+ (n-2)\int_M \langle X,\nabla \lambda\rangle \, dM.
\end{equation*}

Thus,
\begin{equation*}
\int_M |\nabla X|^2 \, dM
= \left(\frac{m}{2} + n - 4 \right)
\int_M \langle X,\nabla \lambda\rangle \, dM.
\end{equation*}

Hence, under all assumptions considered, we conclude that
$
\int_M |\nabla X|^2\, dM \leq 0.
$
So, it follows that $|\nabla X|=0$, and therefore $X$ is parallel, which yields the result.
\end{proof}
 We now recall some notions that will be useful in the sequel. According to \cite{StepanovShandra2003}, a vector field $X$ on $(M^n,g)$ is an infinitesimal harmonic transformation (in the sense recalled in the Introduction) if and only if $\operatorname{tr}_g (\mathcal{L}_X\nabla)=0$, where $\nabla$ is the Levi-Civita connection of $M^n$ and $\mathcal{L}_X\nabla$ is the symmetric $(1,2)$-tensor field $(\mathcal{L}_X\nabla)(U,V)=[X,\nabla_UV]-\nabla_{[X,U]}V-\nabla_U[X,V]$. It should be noted that this property is equivalent to the fact that $X$ defines a harmonic section of the tangent bundle (see also, for example, \cite{nouhaud1972transformations}). In turn, the definition of the rough Laplacian operator $\bar{\Delta}$ acting on vector fields is given by $\bar{\Delta}X=-\operatorname{tr}_g(\nabla^2X)$. Explicitly, if $\{E_i\}$ is a local orthonormal frame field on $M^n$, then
\begin{equation}\label{eq:rough-Laplacian}
\bar{\Delta}X = \sum_i\left( \nabla_{\nabla_{E_i}E_i}X -\nabla_{E_i}\nabla_{E_i}X \right).
\end{equation}

Next, we shall also denote by $\operatorname{Ric}$ the $(1,1)$-tensor associated with the Ricci tensor, that is,
$
\operatorname{Ric}(U,V)=\langle \operatorname{Ric}(U),V\rangle,
$
for all $U,V\in \mathfrak{X}(M)$.

In \cite{StepanovShandra2003}, the authors proved that a vector field $X$ generates an infinitesimal harmonic transformation on a Riemannian manifold $(M^{n},g)$ if and only if
$ 
\Delta X = 2\operatorname{Ric}(X),
$
where $\Delta$ stands for the Hodge Laplacian (acting on $X^\flat$), related to the rough Laplacian through the Weitzenböck formula
$
\Delta X=\bar{\Delta}X+\operatorname{Ric}(X).
$
Equivalently, $X$ is an infinitesimal harmonic transformation if and only if $\bar\Delta X=\Ric X$.

We close this section by recalling the classical \emph{Yano formula}.
Let $(M^n,g)$ be a closed Riemannian manifold. 
Then, for any smooth vector field $Y$ on $M^n$, Yano’s identity (see \cite[p.~170]{Poor1981}) reads as
\begin{equation}\label{eq:yano}
\frac{1}{2}\int_M |\mathcal{L}_Y g|^{2}\,dM
=
\int_M \big(
|\nabla Y|^{2}
+(\operatorname{div}Y)^{2}
-\operatorname{Ric}(Y,Y)
\big)\,dM .
\end{equation}

\section{Proofs of The Results}\label{sec:proofs}

This section is dedicated to proving our main results.

\begin{proof}[Proof of Theorem~\ref{theoA}]

First, we may use Lemma~\ref{lem:laplacianX} to rewrite
the term $|\nabla X|^{2}-\operatorname{Ric}(X,X)$ appearing in \eqref{eq:yano}.
Indeed, we have
$$
|\nabla X|^{2}-\operatorname{Ric}(X,X)
=
\frac{1}{2}\,\Delta |X|^{2}
-\frac{2}{m}|X|^{2}\operatorname{div}X
-(2-n)\langle X,\nabla\lambda\rangle .
$$

Integrating \eqref{eq:yano} with $Y=X$, and using the identity above, yields
$$
\frac{1}{2}\int_M |\mathcal{L}_X g|^{2}\,dM
=
\int_M \Big(
(\operatorname{div}X)^{2}
-\frac{2}{m}|X|^{2}\operatorname{div}X
-(2-n)\langle X,\nabla\lambda\rangle
+\frac{1}{2}\Delta |X|^{2}
\Big)\,dM .
$$

Because $M^n$ is closed, the integral of $\Delta |X|^{2}$ vanishes, and we get
\begin{equation}\label{eq27.2}
\frac{1}{2}\int_M |\mathcal{L}_X g|^{2}\,dM
=
\int_M \Big(
(\operatorname{div}X)^{2}
-\frac{2}{m}|X|^{2}\operatorname{div}X
-(2-n)\langle X,\nabla\lambda\rangle
\Big)\,dM .
\end{equation}

Applying the gradient to equation~\eqref{eq2}, one has
$$
\nabla S+\nabla(\operatorname{div}X)
=n\nabla\lambda +\frac{1}{m}\nabla(|X|^{2}).
$$

Taking the inner product with $X$,
$$
\langle\nabla S, X\rangle
+\langle\nabla(\operatorname{div}X),X\rangle
=\frac{1}{m}\langle\nabla(|X|^{2}),X\rangle
+n\langle X,\nabla\lambda\rangle.
$$

Furthermore,
$$
\operatorname{div}(|X|^{2}X)
=|X|^{2}\operatorname{div}X
+2\langle\nabla_{X}X,X\rangle,
$$
$$
\operatorname{div}\big((\operatorname{div}X)X\big)
=(\operatorname{div}X)^{2}
+\langle\nabla(\operatorname{div}X),X\rangle.
$$

Integrating $\operatorname{div}\big((\operatorname{div}X)X\big)$ over
$M^n$ and applying the divergence theorem,
$$
0=\int_M \Big(
(\operatorname{div}X)^{2}
+\frac{1}{m}\langle\nabla(|X|^{2}),X\rangle
+n\langle X,\nabla\lambda\rangle-\langle\nabla S, X\rangle 
\Big)\,dM .
$$

Similarly, integrating $\operatorname{div}(|X|^{2}X)$,
$$
\int_M \langle\nabla(|X|^{2}),X\rangle\,dM
=-\int_M |X|^{2}\operatorname{div}X\,dM .
$$

Consequently, we arrive at
\begin{equation}\label{eq27}
\int_M |X|^{2}\operatorname{div}X\,dM
=
m\int_M ((\operatorname{div}X)^{2}\
+n\langle X,\nabla\lambda\rangle-\langle\nabla S, X\rangle)\,dM .
\end{equation}

Combining \eqref{eq27} and \eqref{eq27.2},
we obtain
$$
\begin{aligned}
\frac{1}{2}\int_M |\mathcal{L}_X g|^{2}\,dM
&=
\int_M \Big(
(\operatorname{div}X)^{2}
-2(\operatorname{div}X)^{2}
-2n\langle X,\nabla\lambda\rangle
-(2-n)\langle X,\nabla\lambda\rangle
+2\langle\nabla S, X\rangle\Big)\,dM \\
&=
-\int_M \Big(
(\operatorname{div}X)^{2}
+(n+2)\langle X,\nabla\lambda\rangle
-2\langle\nabla S, X\rangle\Big)\,dM .
\end{aligned}
$$

The previous identity and the assumptions $\int_M \langle X,\nabla\lambda\rangle\, dM\ge0$ and $\int_M \langle X,\nabla S\rangle\, dM \le0$ give the following
$$
\int_M |\mathcal{L}_X g|^{2}\,dM \le 0.
$$

Therefore, $\mathcal{L}_X g=0$, and $X$ is a Killing vector field. So, the proof is completed.
\end{proof}

We now turn to the proof of Proposition \ref{prop-div}.
\begin{proof}[Proof of Proposition~\ref{prop-div}]
Since $\operatorname{div}X=0$, integration by parts gives
$$
\int_M\langle X,\nabla S\rangle\,dM
=
\int_M\langle X,\nabla\lambda\rangle\,dM
=0.
$$
Thus, Theorem~\ref{theoA} implies that $X$ is Killing.
For completeness, we also recall Ghosh's direct argument. Assuming that $\operatorname{div} X = 0$,
we apply Lemma~\ref{lem:laplacianX} to obtain
$$
\frac{1}{2}\,\Delta |X|^{2}
= |\nabla X|^{2}
- \operatorname{Ric}(X,X)
+ (2-n)\,\langle X,\nabla \lambda\rangle .
$$

Integrating this identity over $M^n$ and comparing it with the formula
\eqref{eq:yano}, we deduce that
$$
\frac{1}{2}\int_M | \mathcal{L}_X g |^2 \, dM = (n-2) \int_M\,\langle X,\nabla \lambda\rangle\, dM.
$$

Since $\operatorname{div}X=0$, the divergence theorem implies
$$
0 = \int_M \operatorname{div}(\lambda X)\, dM
= \int_M \langle X,\nabla \lambda \rangle \, dM .
$$

Hence, $\mathcal{L}_X g =0$, so $X$ is a Killing vector field.
\end{proof}

We are now in a position to prove Corollary~E.

\begin{proof}[Proof of Corollary~\ref{corE}]

Evaluating \eqref{eq:quasi-einstein} along the pair $(X,X)$,
$$
\operatorname{Ric}(X,X) + \langle \nabla_X X, X \rangle
= \frac{1}{m}\lvert X\rvert^4 + \lambda \lvert X\rvert^2 .
$$

On the other hand,
$$
\operatorname{div}(|X|^2 X)
= \langle \nabla |X|^2, X \rangle + |X|^2 \operatorname{div} X.
$$

Since $\operatorname{div}X=0$, integration over $M^n$ together with the divergence theorem gives
$$
\int_M \langle \nabla_X X, X \rangle \, dM = 0 .
$$

Therefore,
$$
\int_M \operatorname{Ric}(X,X) \, dM
= \int_M\left(\frac{1}{m}\lvert X\rvert^4 + \lambda \lvert X\rvert^2\right) dM .
$$

Under assumption $\operatorname{div}X=0$, Lemma~\ref{lem:laplacianX} reduces to
$$
\frac{1}{2}\,\Delta |X|^{2}
= |\nabla X|^{2}
- \operatorname{Ric}(X,X)
+ (2-n)\,\langle X,\nabla \lambda\rangle .
$$

Integrating over $M^n$ and using the divergence theorem
$$
\int_M \left(
|\nabla X|^2
- \operatorname{Ric}(X,X)
+ (2-n)\,\langle X,\nabla \lambda\rangle
\right) dM = 0 .
$$

Combining this identity with the previous integral formula for
$\int_M \operatorname{Ric}(X,X)\, dM$, we conclude that
$$
\int_M \left(
|\nabla X|^2
- \frac{1}{m}\lvert X\rvert^4
- \lambda \lvert X\rvert^2
+ (2-n)\,\langle X,\nabla \lambda\rangle
\right) dM = 0 .
$$

Using again $\operatorname{div}X=0$, the integration by parts gives
$$
0 = \int_M \operatorname{div}(\lambda X)\, dM
= \int_M \langle X,\nabla \lambda \rangle \, dM .
$$

Hence,
$$
\int_M
|\nabla X|^2 \, dM=\int_M
\left( \frac{1}{m}\lvert X\rvert^4
 +\lambda \lvert X\rvert^2\right)\, dM.
$$

Under assumptions $m<0$ and $\lambda \le 0$, the right-hand side is non-positive,
and therefore $\int_M
|\nabla X|^2\,dM=0$. Thus $X$ is parallel. Since, in the preceding identity, both terms on the right-hand side are nonpositive for $m<0$ and $\lambda\leq 0$, they must vanish identically. Hence $X\equiv0$.
\end{proof}

We now present the identity below, which is interesting in its own right and will play an important role in the proof of Theorem \ref{teo6.1}.

\begin{lemma}\label{eq10}
Let $(M^n,g,X,\lambda)$ be a generalized $m$-quasi-Einstein manifold. Then, for every smooth vector field $Z$ on $M$, the following identity holds:
\begin{align}
\langle 2\Ric X - \Delta X,Z\rangle
&=
\frac{2}{m}\Bigl(
\operatorname{div}(X)\,\langle Z,X\rangle
+\langle \nabla_XX,Z\rangle
-\langle \nabla_ZX,X\rangle
\Bigr) \notag\\
&\qquad
-(n-2)\langle \nabla\lambda,Z\rangle, \label{neweq10.1}
\end{align}
\end{lemma}

\begin{proof} The identity \eqref{neweq10.1} can be viewed as an extension of Equation~(5.4) in Ghosh~\cite{ghosh} to the generalized setting. Although the proof follows essentially the same lines and also relies on Yano's commutation formula, we include it here for completeness and convenience of the reader. We also point out that there is a minor typographical error in the corresponding Yano's commutation formula in~\cite{ghosh}.

We first derive an expression for $\mathcal{L}_X\nabla$ that will be used later. Differentiating \eqref{eq:quasi-einstein} along an arbitrary vector field $Z$, we obtain
\begin{align}
(\nabla_Z\mathcal{L}_Xg)(U,V)
+2(\nabla_Z\operatorname{Ric})(U,V)
&=
\frac{2}{m}\Bigl(
\langle U,\nabla_ZX\rangle \langle V,X\rangle
\notag\\
&\qquad\qquad
+\langle U,X\rangle \langle V,\nabla_ZX\rangle
\Bigr)
\notag\\
&\quad
+2(Z\lambda)\langle U,V\rangle .
\end{align}

Using the commutation formula (see Yano \cite{Yan70})
$$
(\mathcal{L}_X\nabla_Z g-\nabla_Z\mathcal{L}_X g-\nabla_{[X,Z]}g)(U,V)
=
-\inner{(\mathcal{L}_X\nabla)(Z,U)}{V}
-\inner{(\mathcal{L}_X\nabla)(Z,V)}{U},
$$
we deduce
\begin{equation}\label{equation3}
\begin{aligned}
&
\langle (\mathcal{L}_X\nabla)(Z,U),V\rangle
+
\langle (\mathcal{L}_X\nabla)(Z,V),U\rangle
\\
&\qquad
=
-2(\nabla_Z\operatorname{Ric})(U,V)
+2(Z\lambda)\langle U,V\rangle
\\
&\qquad\quad
+\frac{2}{m}\Bigl[
\langle U,\nabla_ZX\rangle \langle V,X\rangle
+
\langle U,X\rangle \langle V,\nabla_ZX\rangle
\Bigr].
\end{aligned}
\end{equation}

Do a cyclic permutation of $U,V,Z$ in equation~\eqref{equation3} to obtain the following two additional identities:
\begin{equation}\label{equation4}
\begin{split}
&
\left\langle (\mathcal{L}_X\nabla)(U,V),Z \right\rangle
+
\left\langle (\mathcal{L}_X\nabla)(U,Z),V \right\rangle
\\
&\qquad
=
-2(\nabla_U\operatorname{Ric})(V,Z)
+2(U\lambda)\left\langle V,Z \right\rangle
\\
&\qquad\quad
+\frac{2}{m}\Bigl[
\left\langle V,\nabla_UX \right\rangle
\left\langle Z,X \right\rangle
+
\left\langle V,X \right\rangle
\left\langle Z,\nabla_UX \right\rangle
\Bigr].
\end{split}
\end{equation}

Similarly, we obtain
\begin{equation}\label{equation5}
\begin{split}
&
\left\langle (\mathcal{L}_X\nabla)(V,Z),U \right\rangle
+
\left\langle (\mathcal{L}_X\nabla)(V,U),Z \right\rangle
\\
&\qquad
=
-2(\nabla_V\operatorname{Ric})(Z,U)
+2(V\lambda)\left\langle Z,U \right\rangle
\\
&\qquad\quad
+\frac{2}{m}\Bigl[
\left\langle Z,\nabla_VX \right\rangle
\left\langle U,X \right\rangle
+
\left\langle Z,X \right\rangle
\left\langle U,\nabla_VX \right\rangle
\Bigr].
\end{split}
\end{equation}

Next,  taking into account that $\mathcal{L}_X\nabla$ is symmetric, adding \eqref{equation4} and \eqref{equation5}, and then subtracting \eqref{equation3}, we obtain the following
\begin{equation}\label{eq123}
\begin{split}
\left\langle (\mathcal{L}_X\nabla)(U,V),Z \right\rangle
&=(\nabla_Z\operatorname{Ric})(U,V)
-(\nabla_V\operatorname{Ric})(Z,U)
-(\nabla_U\operatorname{Ric})(V,Z)
\\
&\quad
+\frac1m\Big[
\left\langle Z,X \right\rangle
\Big(
\left\langle \nabla_UX,V \right\rangle
+\left\langle \nabla_VX,U \right\rangle
\Big)
\\
&\qquad\quad
+\left\langle V,X \right\rangle
\Big(
\left\langle \nabla_UX,Z \right\rangle
-\left\langle \nabla_ZX,U \right\rangle
\Big)
\\
&\qquad\quad
+\left\langle U,X \right\rangle
\Big(
\left\langle \nabla_VX,Z \right\rangle
-\left\langle \nabla_ZX,V \right\rangle
\Big)
\Big]
\\
&\quad
-(Z\lambda)\left\langle U,V \right\rangle
+(V\lambda)\left\langle Z,U \right\rangle
+(U\lambda)\left\langle V,Z \right\rangle .
\end{split}
\end{equation}

Thus, taking $U=V=E_i$, where $\{E_1,\dots,E_n\}$ is a local orthonormal frame field, and summing over $i$, the right-hand side of \eqref{eq123} becomes
$$
\frac{2}{m}\left\{
\operatorname{div}(X)\langle Z,X\rangle
+\langle \nabla_XX,Z\rangle
-\langle \nabla_ZX,X\rangle
\right\}
-(n-2)\langle\nabla\lambda,Z\rangle .
$$

Now, let $R$ denote the Riemann curvature tensor of $g$. 
By the well-known identity due to Yano~\cite{Yan70} (see Equation~(5.13), p.~23),

$$
\left\langle (\mathcal{L}_X\nabla)(U,V),Z \right\rangle
=
\left\langle
\nabla_U\nabla_VX
-\nabla_{\nabla_UV}X
+R(X,U)V,
\, Z
\right\rangle,
$$

Hence,
$$
\sum_{i=1}^n
\left\langle
(\mathcal{L}_X\nabla)(E_i,E_i),Z
\right\rangle
=
\sum_{i=1}^n
\left\langle
\nabla_{E_i}\nabla_{E_i}X
-\nabla_{\nabla_{E_i}E_i}X
+R(X,E_i){E_i},
Z
\right\rangle .
$$

It follows immediately from the definitions \eqref{eq:rough-Laplacian} that the left-hand side of \eqref{eq123} may be rewritten as
$
\left\langle
-\bar{\Delta}X+\Ric X,Z
\right\rangle .
$

Using the Weitzenb\"ock formula
$
\Delta X=\bar{\Delta}X+\Ric X,
$
it follows that
$
\langle 2\Ric X - \Delta X,Z\rangle
=
\frac{2}{m}\left\{
\operatorname{div}(X)\langle Z,X\rangle
+\langle \nabla_XX,Z\rangle
-\langle \nabla_ZX,X\rangle
\right\}
-(n-2)\langle\nabla\lambda,Z\rangle,
$ as we wished to prove.
\end{proof}

Now, we are in position to prove Theorem \ref{teo6.1}.

\begin{proof}[Proof of Theorem~\ref{teo6.1}]

Since $X$ is a Killing vector field, equation~\eqref{neweq10.1} in Lemma~\ref{eq10} yields
$ 
0
=
\frac{4}{m}\langle \nabla_XX,Z\rangle
-(n-2)\langle \nabla\lambda,Z\rangle,
$
which implies
\begin{equation}\label{eq104}
\nabla_XX
=
\frac{m}{4}(n-2)\nabla\lambda.
\end{equation}

Moreover, taking the divergence in \eqref{eq:quasi-einstein} and using the contracted second Bianchi identity
$\nabla S = 2\operatorname{div}\operatorname{Ric},
$
we obtain
\begin{equation}\label{eq:div}
\nabla S
=
2\nabla\lambda
+
\frac{2}{m}\nabla_XX,
\end{equation}
where we used that $X$ is Killing. Combining \eqref{eq104} and \eqref{eq:div}, we obtain
$$
\nabla S
=
2\nabla\lambda
+
\frac{n-2}{2}\nabla\lambda
=
\frac{n+2}{2}\nabla\lambda.
$$
Also, since $X$ is Killing,
$$
\nabla|X|^2=-2\nabla_X X
=-\frac{m(n-2)}{2}\nabla\lambda.
$$

If $n=2$, equation~\eqref{eq:quasi-einstein} becomes
$$
\frac1m X^\flat\otimes X^\flat=(K-\lambda)g.
$$
Since the left-hand side has rank at most one, it follows that $X=0$ and $\lambda=K=S/2$.

If $n\geq3$, the identities above show that the constancy of $\lambda$, $S$, and $|X|$, as well as the condition $\nabla_X X=0$, are all equivalent.
This completes the proof.
\end{proof}

We are now ready to establish Theorem \ref{D}.

\begin{proof}[Proof of Theorem~\ref{D}]
Since $X$ is an infinitesimal harmonic transformation and
$\nabla_XX=0$, equality \eqref{neweq10.1} reduces to
\begin{equation*}
\frac{2}{m}\Bigl(
\operatorname{div}(X)\,\langle X,Z\rangle
-\langle \nabla_ZX,X\rangle
\Bigr)=(n-2)\langle \nabla\lambda,Z\rangle.
\end{equation*}

Taking $Z=X$, we obtain
$ 
\frac{2}{m}\operatorname{div}(X)|X|^2=(n-2)X(\lambda).
$ 
If $n=2$ or $X(\lambda)=0$, then it follows that
$\operatorname{div}X=0$ on $\{X\neq0\}$, and hence $\operatorname{div}X=0$ everywhere by continuity. Moreover, using again that $X$ is an infinitesimal harmonic transformation,
Theorem~2.4 in \cite{StepanovShandra2003} yields
$ 
\operatorname{div}\bigl(\mathcal L_Xg-(\operatorname{div}X)g\bigr)=0.
$
So, we conclude that
$\operatorname{div}(\mathcal L_Xg)=0$.

Suppose first that $n\ge3$ and $X(\lambda)=0$.
The scalar curvature variation formula applied along the
local flow of $X$ \cite[Formula~(1.174e)]{Besse}
therefore reduces to
$
X(S)=-\langle\operatorname{Ric},\mathcal L_Xg\rangle.
$
Using the structural equation (\ref{eq:quasi-einstein}), together with
$\operatorname{div}X=0$ and
$(\mathcal L_Xg)(X,X)=2\langle\nabla_XX,X\rangle=0$, we obtain
$$
X(S)=\frac12|\mathcal L_Xg|^2.
$$

On the other hand, the trace of the structural equation \eqref{eq2}
gives $X(S)=0$, since $X(\lambda)=0$ and $X(|X|^2)=0$.
Thus $\mathcal L_Xg=0$.

If $n=2$, the above identity gives $d|X|^2=0$. Hence
$$
(\mathcal L_Xg)(X,\cdot)
=\langle\nabla_XX,\cdot\rangle+\frac12\,d|X|^2=0.
$$

Taking the divergence, we conclude that
$$
0=\operatorname{div}\bigl((\mathcal L_Xg)(X,\cdot)\bigr)
=\bigl(\operatorname{div}(\mathcal L_Xg)\bigr)(X)
+\frac12|\mathcal L_Xg|^2
=\frac12|\mathcal L_Xg|^2.
$$

Therefore, $X$ is Killing in both cases.

Conversely, if $n\ge3$ and $X$ is Killing, Theorem \ref{teo6.1} gives
$$
0=\nabla_XX=\frac{m(n-2)}4\nabla\lambda.
$$

Thus, $\lambda$ is constant and in particular $X(\lambda)=0$, completing the proof of the theorem.
\end{proof}

We are now in a position to present the proof of the following result.

\begin{proof}[Proof of Theorem~\ref{theoF}]
Since $X$ is geodesic, $\langle\nabla_XX,X\rangle=0$ and therefore
$$
|X|^2\operatorname{div}X=\operatorname{div}(|X|^2X)-2\langle X,\nabla_XX\rangle=\operatorname{div}(|X|^2X).
$$

Then, the divergence theorem yields
\begin{equation}\label{eq:F-divzero}
\int_M|X|^2\operatorname{div}X\,dM=0 .
\end{equation}

Therefore, integrating the identity of Lemma~\ref{lem:laplacianX} over $M^n$ and using \eqref{eq:F-divzero}, we infer
\begin{equation}\label{eq:19}
\int_M \left( |\nabla X|^{2} - \operatorname{Ric}(X,X) + (2-n)\,\langle X,\nabla \lambda\rangle \right) dM = 0 .
\end{equation}

Since $\int_M \operatorname{Ric}(X,X) dM \le 0$ and $\int_M \langle X,\nabla \lambda\rangle\, dM \le 0$, we conclude that $\nabla X =0$. Thus, $\operatorname{div} X =0$ and $|X|$ is constant. On the one hand, since $X$ is parallel, it is in particular Killing, so that \eqref{eq104} holds.

On the other hand, as $X$ is geodesic, \eqref{eq104} gives $\frac{m}{4}(n-2)\nabla\lambda=0$, which implies that $\lambda$ is constant for $n>2$. By Lemma~\ref{lem:laplacianX}, we have $\operatorname{Ric}(X,X) = (2-n)\langle X,\nabla\lambda\rangle=0$. So, equation~\eqref{eq:quasi-einstein} then gives $\frac{1}{m}|X|^{4}+\lambda|X|^{2}=0.$

Therefore, if $m$ and $\lambda$ have the same sign, then $X=0$. If $\lambda=0$, then $X=0$ and $(M^n,g)$ is Ricci-flat. Assume now that $X$ is nonzero at some point of $M$. Since $|X|$ is constant, $X$ is nowhere vanishing and hence $\lambda=-\frac{1}{m}|X|^{2}$. This completes the proof.
\end{proof}

\begin{remark}
As observed by Ghosh, the conclusion of Theorem~\ref{theoF} remains valid under the weaker assumption
$$
\int_M
\Big(
\operatorname{Ric}(X,X)
+(n-2)\langle X,\nabla\lambda\rangle
\Big)\,dM
\le 0,
$$
instead of requiring both $\int_M \operatorname{Ric}(X,X) \, dM\le0$ and
$
\int_M \langle X,\nabla\lambda\rangle\, dM\le0. 
$ Indeed, combining this condition with \eqref{eq:19}, one obtains
$ 
\int_M |\nabla X|^2\, dM \le 0,
$
and hence $X$ is parallel. The remainder of the proof is unchanged.
\end{remark}

We close our paper by presenting the following comments. After finishing this work, we became aware of the article \cite{GulerDe2022}.
In that paper, the authors state the following result:

\begin{lemma}\label{lem:gulerde}
Let $(M^n,g,X,\lambda)$ be a generalized quasi-Einstein manifold. Then
\begin{equation}\label{eq:gulerde-lemma1}
- \langle\bar{\Delta} X, Z\rangle + \operatorname{Ric}(X, Z)
= 2\Bigl(\, \langle X, Z\rangle\, \operatorname{div} X
+ \langle\nabla_X X, Z\rangle
- \langle\nabla_Z X, X \rangle\Bigr)
+ (n - 2)\, d\lambda(Z),
\end{equation}
for all $Z \in \mathfrak{X}(M)$.
\end{lemma}
We note that the identity \eqref{eq:gulerde-lemma1} differs from our formula
\eqref{neweq10.1} by the sign of the term $\, \langle\nabla \lambda, Z\rangle$.
The discrepancy appears to arise from a sign error
in Equation~(24) of \cite{GulerDe2022}, which propagates into subsequent computations. We also note that in their setting, the parameter $m$ is assumed to be $1$.

Moreover, the authors use Lemma~\ref{lem:gulerde} to derive the following result:

\begin{theorem}[{\cite[Theorem 2]{GulerDe2022}}]\label{thm:gulerde}
Let $(M^n,g,X,\lambda)$ be a generalized quasi-Einstein manifold such that
the associated $1$-form of the potential vector field $X$ is harmonic.
Then $\Delta X = 4 \operatorname{Ric}(X)$.
\end{theorem}

Since the proof of Theorem~\ref{thm:gulerde} depends directly on Lemma~\ref{lem:gulerde}, the correction of the latter also affects the statement of this theorem. Recall that the $1$-form $X^\flat$ is harmonic when it is closed and co-closed, that is, $dX^\flat=0$ and $\operatorname{div}X=0$; in this case $\Delta X=(d\delta+\delta d)X^\flat=0$ is part of the hypothesis, so the conclusion $\Delta X=4\Ric X$ amounts to $\Ric X=0$. Under these hypotheses, $\langle\nabla_ZX,X\rangle=\langle\nabla_XX,Z\rangle$ for every $Z$ (since $d X^\flat=0$), so the bracket in \eqref{neweq10.1} vanishes and \eqref{neweq10.1} yields
\begin{equation}\label{eq:gulerde-corrected}
\Ric X=-\frac{n-2}{2}\,\nabla\lambda .
\end{equation}

Thus, the correct conclusion of Theorem~\ref{thm:gulerde} is \eqref{eq:gulerde-corrected}, for every value of $m$: the potential vector field belongs to the nullity of the Ricci tensor if and only if $(n-2)\nabla\lambda=0$, that is, when $n=2$ or $\lambda$ is constant, and only in these cases does $\Ric X=0$ hold.

We also point out a further issue in \cite{GulerDe2022}. More precisely, the statement of Theorem 3 does not hold in full generality. The authors state:
\begin{theorem}[{\cite[Theorem 3]{GulerDe2022}}]
Let $(M^n,g,V,\lambda)$ be a closed generalized quasi-Einstein manifold whose potential vector field $V$ is parallel. Then $V=0$.
\end{theorem}

Indeed, if $V$ is parallel, i.e.,
$
\nabla V=0,
$
then $|V|$ is constant and $\operatorname{div}V=0$. Hence, Lemma~\ref{lem:laplacianX} immediately yields equation (44) of \cite{GulerDe2022},
$
\operatorname{Ric}(V,V)+(n-2)\langle\nabla\lambda,V\rangle=0,
$
without the need to integrate or use the compactness assumption.
On the other hand, equation (45) in \cite{GulerDe2022} asserts that
$$
\operatorname{Ric}(X,Y)
=
V^\flat(X)V^\flat(Y)
-
|V|^2 \langle X,Y \rangle,
$$
for all $X,Y\in\mathfrak{X}(M)$. However, this identity does \emph{not} imply that $M^n$ is Ricci-flat. It only shows that $V$ belongs to the nullity of the Ricci tensor, since we have in this case,
$ 
\operatorname{Ric}(V,X)=0
$
for all $X\in\mathfrak{X}(M)$. A similar issue also appears in \cite[Equation (2.6), Lemma 1 item (3)]{BarrosRibeiro2012}, and was already pointed out by the first and third authors in \cite{CarvalhoCosta}. A counterexample is given by the Riemannian product
$
(\mathbb{S}^1\times N^{n-1},\, d\theta^2+g_N),
$ where $(N^{n-1},g_N)$ is a closed Einstein manifold satisfying
$
\operatorname{Ric}_{g_N}=\lambda g_N,
\,\,\lambda$ a constant
with $m\lambda < 0.$ 
In this case, the vector field
$
X=\pm\sqrt{-m\lambda}\,\partial_\theta, 
$
is parallel and
$(\mathbb{S}^1\times N^{n-1},\, d\theta^2+g_N, X, \lambda)$  satisfies \eqref{eq:quasi-einstein},
while clearly
$ 
X\neq0. 
$ Under these assumptions, the example also satisfies the hypotheses
of Theorems~A, D, and~F. Indeed, $\lambda$ and $S=(n-1)\lambda$
are constant, $X$ is parallel (hence geodesic, Killing, and an
infinitesimal harmonic transformation), and $\operatorname{Ric}(X,X)=0$.
Consequently, all the relevant integral conditions hold with equality.
Moreover, the example realizes alternative~(3) of Theorem~F, since
$|X|^2 =-m\lambda$. Moreover, we show in \cite{CarvalhoCosta} that the above counterexample is the only nontrivial one.
Therefore, the conclusion of Theorem 3 does not hold in full generality.

\textbf{Declarations}\\

\textbf{Data Availability} This manuscript has no associated data.\\

\textbf{Conflict of interest} The authors declare that there is no  conflict of interest.\\

\textbf{Acknowledgments} The authors thank A. Ghosh for his kind interest in this work, for clarifying several aspects of his article, which greatly improved the manuscript. We are also grateful to I.~Domingos for his kind interest in this work. We thank the anonymous
referee for a careful reading of the manuscript and for several valuable comments and suggestions, which led to substantial improvements.
The first author thanks the Federal University of Alagoas (UFAL), Arapiraca campus, where part of this work was developed during a research visit. The second author thanks the Department of Mathematics at the Federal University of Pernambuco (UFPE) for its hospitality during his visit, where part of this work was carried out. The third author was supported by CNPq/Brazil Grant 409513/2023-7.
\bibliographystyle{plain}
\bibliography{references}
\end{document}